\documentclass[12pt, reqno, a4paper]{amsart}

\usepackage{amsmath, amssymb, amsthm, amsfonts, enumerate}
\usepackage{latexsym, xcolor, epsfig}
\usepackage{hyperref, mathrsfs}
\usepackage{esint}

\usepackage[numbers, square]{natbib}

\newcommand{\E}{\mathbb{E}} 

\usepackage{scalerel}

\newcommand{\R}{\mathbb{R}}

\newcommand{\F}{\mathbb{F}}

\usepackage{dsfont}\newcommand{\1}{\mathds{1}}
\newcommand{\nd}{\text{ and }}

\newcommand{\supp}{\mathrm{supp}}
\newcommand{\ip}[1]{\left\langle #1 \right\rangle}

\newcommand{\eps}{\varepsilon} 
\renewcommand{\epsilon}{\varepsilon} 
\renewcommand{\leq}{\leqslant} 
\renewcommand{\geq}{\geqslant}

\newtheorem{theorem}{Theorem}[section]
\newtheorem*{theorem*}{Theorem}

\newtheorem*{conjecture*}{Conjecture}

\newtheorem*{proposition*}{Proposition}
\newtheorem{lemma}[theorem]{Lemma}
\newtheorem*{lemma*}{Lemma}

\newtheorem*{corollary*}{Corollary}

\newtheorem*{example*}{Example}

\newtheorem*{claim*}{Claim}

\theoremstyle{definition}
\newtheorem{definition}[theorem]{Definition}
\newtheorem*{definition*}{Definition}

\theoremstyle{remark}

\newtheorem*{remark*}{Remark}

\newtheorem*{question*}{Question}

\numberwithin{equation}{section}

\title{\vspace{-0.7cm}Functions with large additive energy supported on a Hamming Sphere}
\author{James Aaronson}
\email{james.aaronson@maths.ox.ac.uk}

\date{}
\begin{document}

\begin{abstract}
In this note, we prove that, among functions $f$ supported on a Hamming Sphere in $\F_2^n$ with fixed $\ell^2$ norm, the additive energy is maximised when $f$ is constant. This answers a question asked by Kirshner and Samorodnitsky.
\end{abstract}

\maketitle

\section{Introduction}\label{sec:intro}

For a function $f : \F_2^n \rightarrow \R$, we define its Gowers $u_2$ norm to be
\begin{equation}\label{eqn:Gowers_norm_definition}
||f||_{u_2} = \left(\E_{\substack{a_1, a_2, a_3, a_4 \in \F_2^n \\ a_1 + a_2 = a_3 + a_4}} f(a_1) f(a_2) f(a_3) f(a_4)\right)^{1/4}.
\end{equation}
This is also called the \emph{additive energy} of $f$. This agrees with the usual notion of energy for sets (up to a scaling), in the sense that $||\1_{A}||_{u_2}^4 = \frac{E(A)}{N^3}$, where $N = 2^n$ is the size of $\F_2^n$. Similarly, we will define the $\ell^2$ norm to be
\[
||f||_{2} = \left(\E_{a \in \F_2^n} f(a)^2\right)^{1/2}.
\]

For a set $A \subseteq \F_2^n$, define $\mu(A)$ by
\begin{equation}\label{eqn:mu_definition}
\mu(A) = \max_{\substack{f : \F_2^n \rightarrow \R \\ \supp{f} \subseteq A }} \frac{||f||_{u_2}^4}{||f||_2^4},
\end{equation}
where $\supp(f)$ denotes the support of $f$.

Let the Hamming Sphere $S(n, k) \subseteq \F_2^n$ consist of those vectors of weight $k$; in other words, $S(n, k)$ consists of those vectors with exactly $k$ ones.

In \cite{KSpaper}, Kirshner and Samorodnitsky made the following conjecture:
\begin{conjecture*}[Conjecture 1.9 from \cite{KSpaper}]\leavevmode
Let $A = S(n, k)$. Then, $\mu(A) = \frac{1}{N}\frac{E(A)}{|A|^2}$. 

In other words, the ratio $\frac{||f||_{u_2}^4}{||f||_2^4}$ achieves its maximum when $f$ is constant.
\end{conjecture*}
The purpose of this note is to establish this conjecture.
\begin{theorem}\label{thm:the_theorem}\leavevmode
Let $A = S(n, k)$. Then, $\mu(A) = \frac{1}{N}\frac{E(A)}{|A|^2}$.
\end{theorem}

\begin{remark*}
Kirshner and Samorodnitsky define $\mu(A)$ as the maximal value of $\frac{||f||_4^4}{||f||_2^4}$ among functions whose Fourier transform is supported on $A$. However, it can easily be seen that these two formulations are equivalent (up to normalisation) by taking a Fourier transform, and using Parseval's identity and the relation that $||f||_4^4 = ||\hat{f}||_{u_2}^4$.
\end{remark*}

\section{Proof of Theorem \ref{thm:the_theorem}}

Throughout the proof, let $e_1, \dots, e_n$ denote the standard basis for $\F_2^n$, so that any element of $\F_2^n$ may be written $\sum_i \eps_i e_i$, where $\eps_i \in \{0, 1\}$. If $v, w \in \F_2^n$, let $\ip{v, w}$ denote the standard inner product of $v \nd w$. In other words,
\[
\ip{\sum_{i} \eps^{(1)}_i e_i, \sum_{i} \eps^{(2)}_i e_i} = \sum_i \eps^{(1)}_i\eps^{(2)}_i.
\]

Our approach for proving Theorem \ref{thm:the_theorem} is loosely inspired by the idea to consider compressions as in \cite{GreenTaoFreiman}, though the actual compressions we consider are different. 
\begin{definition}\label{defn:compression}
For a function $f : \F_2^n \rightarrow \R$ and $i < j \leq n$, define the $i, j$ compression $f^{(ij)}$ as follows:
\[
f^{(ij)}(x) = 
\begin{cases}
	f(x) & \ip{x, e_i + e_j} = 0 \\
	\sqrt{\frac{f(x)^2 + f(x + e_i + e_j)^2}{2}} & \text{otherwise},
\end{cases}
\]
In other words, let $\pi_{ij} : \F_2^n \rightarrow \F_2^{(n-2)}$ denote the projection given by ignoring the coefficients of $e_i \nd e_j$. Then, $f^{(ij)}(x)$ is the $\ell^2$-average of $f$ over elements of the coset of $\ker \pi_{ij}$ containing $x$, which have the same Hamming weight as $x$.
\end{definition}

The proof of Theorem \ref{thm:the_theorem} relies on the following lemma about compressions.
\begin{lemma}\label{lem:compression}
Let $A = S(n, k)$, and suppose that $f$ is supported on $A$.
\begin{enumerate}
\item\label{item:cl:support} $f^{(ij)}$ is also supported on $A$.

\item\label{item:cl:l2} $||f^{(ij)}||_2 = ||f||_2$.

\item\label{item:cl:u2} $||f^{(ij)}||_{u_2} \geq ||f||_{u_2}$.

\item\label{item:cl:sharp} $||f^{(ij)}||_{u_2} > ||f||_{u_2}$ unless $f = f^{(ij)}$.
\end{enumerate}
\end{lemma}

\begin{proof}
The proofs of (\ref{item:cl:support}) and (\ref{item:cl:l2}) follow immediately from Definition \ref{defn:compression}.

For (\ref{item:cl:u2}), observe that we may rewrite (\ref{eqn:Gowers_norm_definition}) as follows.
\begin{equation}\label{eqn:Gowers_norm_rewritten}
||f||_{u_2}^4 = \frac{1}{4^3}\E_{\substack{b_1, b_2, b_3, b_4 \in \pi_{ij}(\F_2^n) \\ b_1 + b_2 = b_3 + b_4}} \left(\sum_{\substack{a_t \in \pi_{ij}^{-1}(b_i) \cap A \\ a_1 + a_2 = a_3 + a_4}} f(a_1) f(a_2) f(a_3) f(a_4)\right),
\end{equation}
where the outer expectation is over cosets of $\ker \pi_{ij}$, and the factor of $\frac{1}{4^3}$ comes from the fact that we have renormalised the inner expectation to be a summation. Our strategy will be to prove that each bracketed term on the right hand side of (\ref{eqn:Gowers_norm_rewritten}) does not decrease when we pass from $f$ to $f^{(ij)}$.

Observe that, if $f$ is supported on $A$, then the outer expectation of (\ref{eqn:Gowers_norm_rewritten}) may be restricted to terms such that each $b_t$ has Hamming weight either $k, k-1$ or $k-2$, and the size of $\pi_{ij}^{-1}(b_t) \cap A$ depends on whether $b_t$ has weight $k-1$ or not. Thus, we split naturally into three cases.

\emph{Case 1: None of $b_1, b_2, b_3$ or $b_4$ has Hamming weight $k-1$.} 
In this case, the bracketed term is a sum over exactly one term, and is unchanged as we pass from $f$ to $f^{(ij)}$.

\emph{Case 2: Exactly two of $b_1, b_2, b_3 \nd b_4$ have Hamming weight $k-1$.}
Without loss of generality, it is $b_1 \nd b_2$ which have Hamming weight $k-1$. Then, there are two possibilities for the bracketed term, depending on how many of $b_3$ and $b_4$ have weight $k-2$. If neither or both of them do, then we may write the bracketed term as
\[
f(b_1 + e_i) f(b_2 + e_i) f(a_3) f(a_4) + f(b_1 + e_j) f(b_2 + e_j) f(a_3) f(a_4),
\]
where $a_3$ denotes the unique element of $\pi_{ij}^{-1}(b_3) \cap A$ (and likewise for $a_4$). The conclusion then follows from the assertion that 
\[
f(b_1 + e_i) f(b_2 + e_i) + f(b_1 + e_j) f(b_2 + e_j) \leq 2 f^{(ij)}(b_1 + e_i) f^{(ij)}(b_2 + e_i),
\]
which is a consequence of the Cauchy Schwarz inequality. A similar argument applies if exactly one of $b_3 \nd b_4$ have weight $k-2$.

\emph{Case 3: All four of $b_1, b_2, b_3 \nd b_4$ have Hamming weight $k-1$.}
In this case, the bracketed term is now a sum of eight terms. One of the terms is \[f(a_1)f(a_2)f(a_3)f(a_4) = f(b_1 + e_i)f(b_2 + e_i)f(b_3 + e_i)f(b_4 + e_i),\] and the others can be obtained by replacing two or four of the $e_i$ with $e_j$.

Group the terms into four pairs, according to the values of $a_3 \nd a_4$. If $a_3 = b_3 + e_1 \nd a_4 = b_4 + e_1$, for example, then we have
\begin{align*}
f(b_1 + e_i)f(b_2 + e_i)&f(a_3)f(a_4) + f(b_1 + e_j)f(b_2 + e_j)f(a_3)f(a_4) \\
&\leq f^{(ij)}(a_1) f^{(ij)}(a_2)f(a_3)f(a_4),
\end{align*}
as in case 2. The conclusion then follows from the fact that 
\[
f(b_3 + e_i) + f(b_3 + e_j) \leq 2f^{(ij)}(a_3),
\]
which follows from Cauchy-Schwarz.

Finally, it remains to prove (\ref{item:cl:sharp}). But this is easy to do. Suppose that $f \neq f^{(ij)}$; in other words, there is some vector $v$ of weight $k-1$, such that $f(v + e_i) \neq f(v + e_j)$. Then, consider the term of (\ref{eqn:Gowers_norm_rewritten}) coming from $b_1 = b_2 = b_3 = b_4 = v$. It is easy to see that equality will not hold in the relation 
\[\sum f(a_1)f(a_2)f(a_3)f(a_4) \leq 8 f^{(ij)}(a_1)f^{(ij)}(a_2)f^{(ij)}(a_3)f^{(ij)}(a_4). \qedhere \]
\end{proof}

We can now complete the proof of Theorem \ref{thm:the_theorem}. By compactness, there must exist some function $f$ achieving the maximal value of $||f||_{u_2}$, for fixed $||f||_2$. Suppose that this maximal value is achieved for a function $f$ which is not constant. 

Consider the Hamming Sphere as a graph, where we join two elements $v \nd w$ with an edge if and only if $w = v + e_i + e_j$ for some $i \nd j$. Then, the Hamming Sphere is connected. Thus, there must be two adjacent elements $v \nd w$ for which $f(v) \neq f(w)$. 

Thus, if $w = v + e_i + e_j$, then $f \neq f^{(ij)}$, and so Lemma \ref{lem:compression} (\ref{item:cl:sharp}) tells us that $||f||_{u_2} < ||f^{(ij)}||_{u_2}$, contradicting the maximality of $||f||_{u_2}$. 

Therefore, $f$ must be constant, yielding Theorem \ref{thm:the_theorem}. \qed

\addtocontents{toc}{\protect\vspace*{\baselineskip}}


\addtocontents{toc}{\protect\vspace*{\baselineskip}}

\addcontentsline{toc}{section}{References}

\bibliographystyle{plainnat}

\end{document}